\documentclass{article}
\usepackage[english]{babel}
\usepackage{color}
\usepackage{amsmath,amssymb,amsthm}
\usepackage{algorithmic}
\usepackage{algorithm}
\usepackage{graphicx}
\usepackage{epsfig}
\usepackage{subfigure}
\usepackage{multirow}
\usepackage{url}
\usepackage{color, palatino}
\usepackage[top=3cm, bottom=3cm, left=2cm, right=2cm]{geometry}

\newtheorem{definition}{Definition}
\newtheorem{lemma}{Lemma}

\newtheorem{theorem}{Theorem}
\theoremstyle{definition}

\newtheorem{example}{Example}

\newcommand{\new}{\newcommand}
\new{\set}[1]{\{{#1}\}}

\new{\rone}{{\mathbb R}}
\new{\nat}{{\mathbb N}}
\newcommand{\argmin}{\operatornamewithlimits{argmin}}
\newcommand{\argmax}{\operatornamewithlimits{argmax}}
\new{\Prob}[1]{\mathrm{Pr}\left(\, #1 \right)}
\new{\E}{{\mathrm E}}
\new{\eps}{\varepsilon}
\new{\nor}[1]{\left\|{#1}\right\|}
\new{\norr}[1]{\left\|{#1}\right\|_\rho}
\new{\scal}[2]{\langle{#1},{#2}\rangle}
\new{\eqn}[1]{~(\ref{#1})}
\new{\norh}[1]{\left\|{#1}\right\|_\hh}
\new{\noru}[1]{|\!|\!|{#1}|\!|\!|}
\new{\scalh}[2]{\langle{#1},{#2}\rangle_{\hh}}
\new{\scaldue}[2]{\left\langle{#1},{#2}\right\rangle_{\rho}}
\new{\bg}{\begin}

\new{\beeq}[2]{\begin{equation}\label{#1}{#2}\end{equation}}
\new{\room}{\ \ \ \ }

\new{\marg}{\rho_X}
\new{\prob}{\rho}
\new{\dd}{d}
\new{\db}{b}
\new{\tolext}{\nu}
\new{\ovlp}{\alpha}
\new{\ts}{\vz}
\new{\n}{n}
\new{\vz}{{\mathbf z}}
\new{\vx}{{\mathbf x}}
\new{\vy}{{\mathbf y}}
\new{\vs}{{\mathbf s}}
\new{\vv}{{\mathbf v}}
\new{\kk}{K_p^{\mathcal{G}}}
\new{\kknoG}{K_p}
\new{\esp}[1]{{\mathcal E}({#1})}
\new{\emp}[1]{{\mathcal E_\vz}({#1})}
\new{\hh}{{\mathcal H}}
\new{\ldue}{L^2(X,\marg)}
\new{\lp}{L^p(X,\marg)}
\new{\la}{\lambda}
\new{\pone}{\tau}
\new{\ptwo}{\mu}
\new{\tah}{f_\hh}
\new{\regr}{f_\rho}
\new{\fz}{f_\ts}
\new{\fzl}{f_{\ts}^\la}
\new{\phiv}{{\mathcal E}_{\pone}}
\new{\F}{F}
\new{\ff}{f^*}
\new{\acca}{h}
\new{\gr}{G}
\new{\gtot}{\mathcal{G}}
\new{\ngr}{B}
\new{\JG}{\Omega^{\gtot}_p}
\new{\JGtext}{\JG}
\new{\JGtwo}{\Omega^{\gtot}_2}
\new{\proj}[1]{\pi_{#1\kk}}
\new{\projr}[2]{P_{\gtot_{#1},{#2}}}
\new{\ancora}{\lambda}
\new{\som}[2]{\sum_{#1=1}^{#2}}
\new{\K}{{\mathrm K}}
\new{\Z}{{\mathrm Z}}
\new{\LL}{{\mathrm L}}
\new{\St}[1]{\mathbf{S}_{#1}}
\new{\st}[1]{S_{#1}}
\new{\J}{J}
\new{\tv}{\mathcal J}
\new{\FF}{\mathcal F}
\new{\dv}{\textrm{div}}
\new{\w}{x}
\new{\vw}{{\mathbf \w}}
\new{\one}{{\mathbf 1}}
\new{\vf}{\textrm{f}}
\new{\vg}{\textrm{g}}
\new{\dom}{X}
\new{\al}{\alpha}
\new{\be}{\beta}
\new{\id}{I}
\new{\TT}{\mathcal T_\sigma}
\new{\yy}{y}
\new{\B}{B_\sigma}
\new{\prox}{\textrm{prox}}

\new{\jj}{j}
\new{\jjj}{j'}
\new{\jjjj}{\gamma}
\new{\ii}{i}
\new{\iii}{i'}
\new{\iiii}{k}
\new{\itext}{m}
\new{\itint}{l}
\new{\xii}{x_\ii}
\new{\yi}{y_\ii}
\new{\dej}{\partial_\jj}
\new{\step}{\sigma}
\new{\GG}{{\mathcal G}}
\new{\Gk}{{\mathcal G}_k}
\new{\PP}{P}
\new{\kernel}{k}

\new{\note}[1]{$\clubsuit$ \textbf{#1} $\clubsuit$\\}

\new{\blu}[1]{\textcolor{blue}{#1}}

\begin{document}
\title{Proximal methods for the latent group lasso penalty}

\author{\normalsize{Silvia Villa$^\star$,   Lorenzo Rosasco$^\dagger$, Sofia Mosci$^\ddagger$, Alessandro Verri$^\ddagger$ }\\
\small \em $\star$  Istituto Italiano di Tecnologia, Genova, ITALY\\
\small \em $\dagger$ CBCL, McGovern Institute, Artificial Intelligence Lab, BCS, MIT, USA\\
\small \em $\ddagger$ DIBRIS, Universit\`{a} di Genova, ITALY\\
{\small \tt  silvia.villa@iit.it, lrosasco@mit.edu, \{sofia.mosci,alessandro.verri\}@unige.it, }}

\maketitle

\begin{abstract}
We  consider a regularized least squares problem, with regularization by structured sparsity-inducing norms, which extend the usual $\ell_1$  and the group lasso penalty, by allowing the subsets to overlap. Such regularizations lead to nonsmooth problems that are difficult to optimize, and we propose in this paper a suitable  version of an accelerated proximal method to solve them. 
We prove convergence of a nested procedure, obtained composing an accelerated  proximal method with an inner algorithm for computing the proximity operator.
By exploiting the geometrical properties of the penalty, we devise a new active set strategy, 
thanks to which the inner iteration is relatively fast, thus guaranteeing good computational performances of the overall algorithm.  
Our approach allows to deal with high dimensional problems without pre-processing for
 dimensionality reduction, leading to better computational and prediction performances with respect to the state-of-the art methods, as shown empirically 
 both on toy and real data.
\end{abstract}

{\bf keywords:}  Structured sparsity, proximal methods, regularization \\

{\bf AMS Classification:}  65K10, 90C25 \\

\section{Introduction }

Sparsity has become a popular way to deal with a number of problems arising in signal and
image processing, statistics and machine learning \cite{Linz10}. In a broad sense, it refers to the possibility of  writing the solution 
in terms of a few building blocks. Often sparsity based methods are the key towards finding interpretable 
models in  real-world problems.
For example, sparse regularization based with $\ell_1$-type penalties is a powerful approach
to find sparse solutions by  minimizing a convex functional \cite{Tibshirani96, chen1999, efron04}.
The success of  $\ell_1$ regularization motivated exploring different kinds of sparsity properties 
for regularized optimization problems, exploiting available a priori  information, which restricts the admissible sparsity patterns of the solution.
An example of  a sparsity pattern is when  the variables are partitioned into groups (known a priori), 
and the goal is to estimate a sparse model where variables belonging to the same group are either jointly selected
or discarded. This problem can be solved by regularizing with the group $\ell_1$ penalty, also known as group lasso penalty \cite{yuan06}.
The latter is the sum, over the groups, of the euclidean norms of the coefficients restricted to each group. Note that, for any $p>1$, the same groupwise selection can be achieved by regularizing with the $\ell_1$/$\ell_p$ norm, i.e. the sum over the groups of the $\ell_p$ norm of the coefficients restricted to each group.
A possible generalization of the group lasso penalty is obtained considering groups of variables which can be potentially overlapping \cite{zhao08,jenatton2009}, and the goal is to estimate a model which support is the union of groups. 
For example, this is a common situation in bioinformatics (especially in the context of high-throughput data such as gene expression  and mass spectrometry data),
 where problems are characterized by a very low number of samples with several thousands of variables. 
In fact, when the number of samples is not sufficient to guarantee accurate model estimation, a possible solution is to take advantage of
the huge amount of prior knowledge encoded in online databases such as the Gene Ontology \cite{GO2000}. 
Largely  motivated by  applications in bioinformatics, the {\itshape latent group lasso with overlap penalty} is proposed in \cite{jacob2009} and further studied in \cite{obozinski:inria-00628498,bach12} and in \cite{PeyFad11} in the image processing context, which generalizes the $\ell_1$/$\ell_2$ penalty to overlapping groups, thus satisfying the assumption that the admissible sparsity patterns must be unions of a subset of the groups.

All the methods proposed in the literature solve the minimization problem arising in \cite{jacob2009} by applying state-of-the-art techniques for group lasso in an expanded space, called {\em space of latent variables}, built by duplicating variables that belong to more than one group. The most popular optimization strategies that have been proposed are interior-points methods \cite{bach04,park07}, block coordinate descent \cite{meier08}, proximal methods \cite{rosasco09struspa, rosasco2010ecml, PeyFad11, LiuHe10, CheLinKim12} and the related  alternating direction method \cite{DenYinZha11}. Very recently, the paper  \cite{QinGol12} proposed an accelerated alternating direction method and \cite{QinSchGol12} studied a block coordinate descent, along with a proximal method with variable step-sizes.

As already noted in \cite{jacob2009}, though very  natural, every implementation developed in the latent variables does not scale to large datasets: 
when the groups have significant overlap, a more scalable algorithm with no data duplication is needed. 
For this reason we propose an alternative optimization approach to solve the group lasso problem with overlap, and extend it to the
entire family of group lasso with overlap penalties,  that generalize the 
$\ell_1$/$\ell_p$ penalties to overlapping groups for $p>1$.
Our method is a two-loops iterative scheme based on proximal methods (see for example \cite{nesterov83, becker09, beck09}), and more precisely on the accelerated version named FISTA \cite{beck09}. 
It does not require explicit replication of the variables and is thus more appropriate to deal with high dimensional problems with large group overlap.
In fact, the proximity operator can be efficiently computed by exploiting the geometrical properties of the penalty. We show that such an operator can be written as the identity minus the projection onto a suitable convex set, which is the intersection of as many convex sets as the number of {\itshape active groups}, that is groups corresponding to active constraints, which can be easily found. Indeed, the identification of the active groups is a key step, since it allows  computing the projection in a reduced space.
For general $p$, the projection can  be solved via the Cyclic Projections algorithm \cite{bauschke94ancora}.
Furthermore, for the case $p=2$, we present an accelerated scheme, where the reduced projection is computed by solving a corresponding  dual problem via the projected Newton method \cite{bertsekas82}, thus working in a much lower dimensional space.

The present paper completes and extends the preliminary results presented in the short conference version \cite{glopridu}. 
In particular, it contains a general mathematical presentation and all the proofs, which were omitted in \cite{glopridu}. We next describe how the rest of the paper is organized, and then highlight the main novelties with respect to the short version. 
In Section \ref{sec:problem}, we cast the problem of Group-wise Selection with Overlap (GSO)
as a regularization problem based on a  modified $\ell_1$/$\ell_p$-type penalty and compare it with other structured sparsity penalties.
We extend the approach in \cite{glopridu} for $p=2$ to general $p>1$.
In Section \ref{sec:algo}, we describe the derivation of the proposed optimization scheme, 
and prove its convergence.
Precisely, we first recall proximal methods in Subsection \ref{sec:ISTA}, 
then in Subsection \ref{sec:prox} we describe the technical  results that ease the computation 
of the proximity operator as a simplified projection, 
and present different projection algorithms depending on $p$. 
With respect to \cite{glopridu}, we show that our active set strategy can be profitably used 
in this generalized framework in combination with any algorithm chosen to compute the inner 
projection. Furthermore, to solve the projection for a general $p\in(1,+\infty]$, we discuss the use of a cyclic 
projections algorithm, whose convergence in norm is guaranteed and results in a rate of convergence 
for the proposed proximal method, proved in Subsection \ref{sec:convergence}.
Section \ref{sec:expe} is a substantial extension of the experiments performed in \cite{glopridu}. 
We empirically analyze the computational performance of our optimization procedure.
We first study the performance of the different variations of the proposed optimization scheme.
Then we present a set of numerical experiments comparing  running time of our algorithm with  state-of-the-art techniques. 
We conclude with a real data experiment where we show that the improved computational performance
allows dealing with large data sets without preprocessing thus improving also
the prediction and selection performance. Finally, in Appendix B 
we review the projected Newton method \cite{bertsekas82}.\\

\noindent{\bf Notation}. Given a vector $ x\in\rone^\dd$, 
we denote with  $\nor{\cdot}_p$  the $\ell_p$-norm of $ x$,  defined as $\nor{ x}_p = (\sum_{j=1}^\dd  x_j^p)^{1/p}$
and  $\nor{\w}_{\infty} = \max_{j\in\{1,\dots\dd\}} |\w_j|$.
We will also use the notation $\nor{\w}_{\gr,p} = (\sum_{j \in \gr}\w_j^p)^{1/p}$ for $p\geq 1$, and
 $\nor{\w}_{\gr,\infty} = \max_{j\in\gr} |\w_j|$
to denote the $\ell_p$-norm of the components of $\w$ in $\gr\subset\{1,\dots,\dd\}$. 
When the subscript $p$ is omitted, the $\ell_2$ norm is used, $\nor{\cdot}=\nor{\cdot}_2$. 
The conjugate exponent of $p$ is denoted by $q$; we recall that $q$ is such that $1/p+1/q=1$.
In the following, $X$ will denote $\rone^\dd$ and $Y$ a bounded interval in $\rone$.

\section{Group-wise selection with Overlap (GSO)}\label{sec:problem}
This paper proposes an optimization algorithm for a regularized least-squares problem of the type
\begin{equation*}\label{GSO}\tag{GSO-$p$}
\min_{ x \in\rone^\dd}\   \phiv^p(\w), \qquad \phiv^p(\w) =\frac{1}{n} \nor{\Psi x-y}^2+2\pone\JG(\w)\, , 
\end{equation*}
where $\Psi:\rone^\dd\to\rone^\n$ is a linear operator, $y\in\rone^n$, and  $\JG:\rone^\dd\to[0,+\infty)$ is a convex and lower semicontinuous penalty, depending on a parameter $p\in (1,+\infty]$, and on an a priori given group structure, $\mathcal{G} = \{\gr_r\}_{r=1}^\ngr$,  with $\gr_r\subset\{1,\dots,\dd\}$ and $\bigcup_{r=1}^\ngr \gr_r = \{1,\dots,\dd\}$. Note that other data fit terms could be used, different from the quadratic one, as long as they are convex and continuously differentiable with Lipschitz continuous gradient. We will focus on least squares to simplify the exposition. 
Most group sparsity penalties can be built starting from the family of canonical linear projections on the subspace identified by the indices belonging to $G_r$, i.e. $P_r:\rone^\dd\to \rone^{G_r}$. 
The definition of the penalties we consider is based on the adjoint of the linear operator
\[
P:\rone^\dd \to\prod_{r=1}^\ngr \rone^{G_r}, \quad P\w=(P_1\w,\ldots,P_B\w) ,
\]
that is the operator
\[
P^*:\prod_{r=1}^\ngr \rone^{G_r}\to \rone^\dd, \quad P^*(v_1,\dots,v_B)=\sum_{r=1}^\ngr P_r^*v_r,
\] 
where $P_r^*:  \rone^{G_r}\to \rone^\dd$ is the canonical injection.  For $\w\in\rone^\dd$ we set 
\beeq{penalty}{ \JG(\w) = \min_{\substack{v\in\prod \rone^{G_r}\\ P^*v=\w}} \sum_{r=1}^\ngr \nor{v_r}_p.} 
For $p=2$, the functional $\JGtwo$ was introduced in \cite{jacob2009} (see also \cite{obozinski:inria-00628498,bach12}). The distinctive feature of the family of penalties $\JG$, is that they have the  property of inducing group-wise selection, that is they lead to solutions with support (i.e. set of non zero entries) which is the {\itshape union} of a subsets of the groups defined a priori. In fact, $\JG$ can be seen as a generalization of the mixed $\ell_1/\ell_p$ norms, originally introduced for disjoint groups:
\[
\mathrm{R}^{\mathcal{G}}_p(\w)=\sum_{r=1}^\ngr  \nor{\w}_{\gr_r,p}\,, \qquad p\geq 1.
\] 
For $p=2$, $\mathrm{R}^{\mathcal{G}}_p$ is the group lasso penalty, and it is well-known \cite{yuan06} that such penalties lead to solutions whose support  is the union of a small number of groups. The penalty $\mathrm{R}^{\mathcal{G}}_p$ can be written  also if the groups overlap, and more generally the composite absolute penalties (CAP)
$$
J^{\mathcal{G}}_{\gamma,p}(\w)=\sum_{r=1}^\ngr  (\nor{\w}_{\gr_r,p})^\gamma\,,
$$
first introduced in \cite{zhao08} and coinciding with  $\mathrm{R}^{\mathcal{G}}_p$ for $\gamma=1$, have been intensively studied. The $J_{\gamma,p}$ penalties allow to deal with  complex groups structures involving  hierarchies or graphs  and it is proved in  \cite{jenatton2009} that the CAP penalties constraint the support to be the {\em complement of a union} of groups. $\JG$ and $\mathrm{R}^{\mathcal{G}}_p$ are thus somehow complementary and have different domain of applications \cite{jenatton2009,mairal2010network,jenatton2010SPCA}. 

While many algorithms have been proposed to solve the optimization problem corresponding to $\mathrm{R}^{\mathcal{G}}_p$, the one corresponding
$\JG$ is much less studied. This is due on the one hand to the fact that the penalty is more complex, and on the other hand to the widespread use of the ``replication strategy''. The latter is based on the observation that, using the definition of $\JG$, and the surjectivity of $P^*$, the (GSO-$p$) minimization problem   can be written as
\beeq{eq:repl}{
\min_{v\in \prod_{r=1}^\ngr \rone^{G_r}} \frac{1}{n} \nor{\Psi P^*v-y}^2+2\pone  \sum_{r=1}^\ngr \nor{v_r}_p,}
which is a group lasso problem without overlap for the linear operator $\Psi P^*$  in the so called {\em latent variables} $(v_r)_{r=1}^B$, obtained by replicating variables belonging to more than one group. The last rewriting allows to apply every algorithm developed for the standard group-lasso to the overlapping case, but this strategy is not feasible for high dimensional problems with large group overlaps, as potentially many artificial dimensions are created. The main goal of this paper is to propose and study an optimization algorithm which does not require the replication of variables belonging to more than one group.

The choice $p>1$ has both technical and practical motivations.
On the one hand, it guarantees convexity of the penalty -- which can be shown to be a norm (see Lemma 1 in  \cite{jacob2009} for $p=2$) --, and, as a consequence, of the \eqref{GSO} regularization problem  (note that this is valid for $p=1$ too).
On the other hand, it enforces  {\itshape ``democracy''} among the elements that belong to the same group, in the sense that no intragroup sparsity is enforced, thus inducing group-wise selection.
The case $p=1$ is trivial, since the penalty $\Omega^{\mathcal{G}}_1$  coincides with the $\ell_1$ norm, or lasso penalty \cite{Tibshirani96}:
$$
\Omega^{\mathcal{G}}_1(\w)=\!\!\!
\inf_{\substack{(v_1,\dots,v_\ngr)\in \prod \rone^{G_r} \\ P^*v= x}}
\sum_{r=1}^\ngr\sum_{j\in\gr_r}\!\!  |(v_r)_j|
 =\!\! \!
\inf_{\substack{(v_1,\dots,v_\ngr)\in \prod \rone^{G_r} \\ P^*v= x}}
\sum_{j=1}^\dd \sum_{r: j\in\gr_r}^\ngr\!\!\!   |(v_r)_j|
=\sum_{j=1}^\dd|\w_j|,
$$ 
and is thus independent of $\mathcal{G}$.

\begin{example}  A particular instance of the above problem occurs in statistical learning. Assume that the estimator and the regression function can be described by a generalized linear model $f(x)=\sum_{j=1}^\dd  x_j\psi_j(x)$, for a given dictionary $\{\psi_j\}_{j=1}^\dd$ of functions $\psi_j: X\to Y$ (with $X$ a set and $Y\subseteq \rone$). Given a training set $\{(x_i,y_i)_{i=1}^\n\} \in (X\times Y)^\n$ the regularized empirical risk takes the form
\[
\frac{1}{n} \nor{\Psi x-y}^2+2\pone\JG(\w),
\]
with $\Psi:\rone^\dd\to\rone^\n$,   $\Psi\w=\sum_{j=1}^\dd \psi_j(x_i) x_j$ and $y=(y_1,\ldots,y_n)$. 
\end{example}

\begin{example} Most results obtained in the paper hold in an infinite dimensional setting. In particular, our approach can be naturally extended to the {\em multiple kernel learning}(MKL) problem \cite{michelli05}.  For this problem, given reproducing kernel Hilbert spaces $\hh_1,\ldots,\hh_m$ of functions $g:\mathcal{X}\to \mathbb{R}$, defining $\hh=\sum_{r=1}^m \hh_r$, the resulting optimization problem takes the form (see \cite{michelli05}) 
\[
\min_{\substack{g\in \prod_r\hh_{r}} } \nor{\Psi (\sum_r g_r)-y}^2+ \sum_{r=1}^m \|g_r\|_{\hh_r},
\]
for a suitable $\Psi:\hh \to \mathbb{R}^n$, $y\in\mathbb{R}^n$. As can be readily seen, the multiple kernel learning problem has the same structure of the \eqref{GSO} problem described above. 
\end{example}


\section{An efficient proximal algorithm}\label{sec:algo}

Due to non-smoothness of the penalty term, solving the \eqref{GSO} minimization problem  is not trivial.
Moreover, if one needs to solve the \eqref{GSO} problem  for high dimensional data, the use of standard second-order methods such as interior-point methods is precluded  (see for instance \cite{becker09}), since they need to solve large systems of linear equations to compute the 
Newton steps.  On the other hand,  first order methods inspired to Nesterov's seminal paper \cite{nesterov05} (see also \cite{nesterov83})
and based on the proximal map are accurate, and robust, in the sense that their performance does not depend on the fine tuning of various controlling parameters. 
Furthermore, these methods  were already  proved to be a computationally  efficient alternative for solving many regularized inverse problems in image processing \cite{ChaPoc11}, compressed sensing \cite{becker09} and machine  learning applications \cite{bach12,duchi2009,rosasco2010ecml}.  

\subsection{Proximal methods}\label{sec:ISTA}

The (GSO-$p$) regularized convex functional is the sum of a convex smooth term, $F(\w)=\frac{1}{n} \nor{\Psi\w-y}^2$, with Lipschitz continuous gradient,  and a non-differentiable penalty $\pone\JG(\cdot)$. A minimizing sequence can be computed with a proximal gradient algorithm \cite{Tse11} (a.k.a. forward-backward splitting method \cite{combettes}, and Iterative Shrinkage Thresholding Algorithm (ISTA) \cite{beck09})
\begin{align*}\label{ISTA}\tag{ISTA}
\w^\itext = \textrm{prox}_{\frac{\pone}{\sigma}\JG}\left(\w^{\itext-1}-\frac{1}{2\sigma}
\nabla F(\w^{\itext-1})\right)
\end{align*}
for a suitable choice of $\sigma$, and  any initialization $\w^0$.  Recently, several accelerations of ISTA have been proposed \cite{Nes07,Tse11,beck09}. With respect to ISTA, they only require the additional computation of a linear combination of two consecutive iterates. Among them, FISTA (Fast Iterative Shrinkage Thresholding Algorithm) \cite{beck09} is given by the following updating rule for $\itext\geq 1$
\begin{align}\label{FISTA}
\nonumber \w^\itext &= \textrm{prox}_{\frac{\pone}{\sigma}\JG}\left(\acca^{\itext}-\frac{1}{2\sigma}
\nabla F(\acca^{\itext})\right)\\
\tag{FISTA}
s_{\itext+1} &= \frac{1}{2}\left(1+\sqrt{1+4s_\itext^2}\right)\\
\nonumber\acca^{\itext+1} &= \left(1+\frac{s_{\itext}-1}{s_{\itext+1}} \right)  \w^{\itext}+ \frac{1-s_{\itext}}{s_{\itext+1 }}\w^{\itext-1}
\end{align}
for a suitable choice of $\sigma>0$, $s_1 = 1$, and  any initialization $\acca^1=\w^0$. 
Both schemes are based on the computation of the proximity operator \cite{Mor62}, which is defined as  
\begin{equation}\label{eq:def_prox}
\textrm{prox}_{\lambda \JG} (z)=\argmin_{x\in \rone^\dd} \Phi_\lambda(x),\qquad\text{with}\quad \Phi_\lambda(x)=\frac1{2\lambda}\nor{x-z}^2+\JG(x) ,\qquad \lambda>0.
\end{equation}

\begin{algorithm}[!h]
\caption{FISTA for GSO-$p$}
\label{algo:GSO}
\begin{algorithmic}
\STATE \textbf{Given:} $\mathcal{G}, ~p\in(1,+\infty],~\tau>0, ~ \epsilon_0>0,\alpha>0,~\w^0=\acca^0 \in\rone^\dd, s_0=1, $ \\
\STATE  \textbf{Let:}  $\sigma = ||\Psi^T\Psi||/n, \itext = 0$ and $q$ such that $\frac 1 p +\frac 1 q =1$.
\WHILE{\texttt{convergence not reached}}
\STATE \vspace{0.1cm}
\begin{itemize}
\item  $\hat{\acca}^\itext = \acca^{\itext}-\frac{1}{\n\sigma} \Psi^T(\Psi \acca^{\itext}-y)$\\
\item Find $\hat{\mathcal{G}}^\itext = \{\gr\in\mathcal{G},\nor{\hat{\acca}^\itext}_\gr\geq\frac{\tau}{\sigma}\}$\\
\item Approximately compute the projection of $\hat{\acca}^\itext$ onto 
$\frac\tau\sigma K_p^{\hat{\mathcal{G}}^m}:=\bigcap_{\gr\in\hat{\mathcal{G}}^\itext}  \left\{\acca\in\rone^\dd\,:\, \nor{\acca}_{\gr,q}\leq\frac{\pone}{\sigma}\right\}$
with tolerance $\epsilon_0\itext^{-\alpha}$  \\
\item $\w^\itext = \hat{\acca}^\itext-\pi_{\frac{\tau}{\sigma}K_p^{\hat{\mathcal{G}}^\itext}}(\hat{\acca}^\itext)$
\item $s_{\itext+1} = \frac{1}{2}\left(1+\sqrt{1+4s_\itext^2}\right)$ \\
\item $\acca^{\itext+1}  =\left(1+\frac{s_{\itext}-1}{s_{\itext+1}} \right)\w^{\itext} +\frac{1-s_{\itext}}{s_{\itext+1 }} \w^{\itext-1}$\\
\end{itemize}
\ENDWHILE
\RETURN $\w^\itext$
\end{algorithmic}
\vspace{.3cm}
\end{algorithm}

The  convergence rate of $\phiv^p( x^\itext) -\min \phiv^p$,  for ISTA  and FISTA, is $O(1/\itext)$ and $O(1/\itext^2)$, respectively, 
when the proximity operator is computed exactly. However, in general, the exact expression is not available. 
Recently, it has been shown that, also in the presence
of  errors, the accelerated version maintains advantages with respect to the basic one.  
In fact, the rate $O(1/\itext^2)$ for FISTA in the presence of computational errors was recently proved in \cite{Bach11, salzo2011prox}
for various error criteria. Convergence of  ISTA with errors was already known, and first 
proved in  \cite{Roc76, combettes}.

Since the proximity operator of the penalty $\JG$ is not admissible in closed form, 
 the \eqref{GSO} minimization problem can thus be solved via an inexact version 
 of the iterative schemes \ref{ISTA} or \ref{FISTA},  where $\nabla F(\acca^\itext)$ is simply $2\Psi^T(\Psi \acca^\itext - y)/n$. 
Note that, in the special case of not overlapping groups, the proximity operator can be explicitly evaluated group-wise, and reduces to a group-wise soft-thresholding operator. In the general case, as explained in Subsection \ref{sec:prox}, the proximity operator can be written in terms of a projection, and  we will provide an algorithm to approximately compute it. Note also that we will show that at each step the projection involves only a subset of the initial groups, the active groups, thus significantly increasing the computational performance of the overall algorithm.\\

\subsection{Computing the proximity operator of $\JG$}\label{sec:prox}
In this subsection we state the lemmas that allow us to efficiently compute the proximity operator of $\JG$ and to formulate the inexact version of FISTA reported in Algorithm \ref{algo:GSO}.

As a direct consequence of standard results of convex analysis, Lemma \ref{lemma:prox} shows that the computation of the proximity operator amounts to the computation of a projection operator onto the intersection of convex sets, each of them corresponding to a group. 
In Lemma \ref{lemma:proj}, we theoretically justifies an active set strategy, by showing that when projecting a vector onto this intersection, it is possible to discard the constraints which are already satisfied.

\begin{lemma}
\label{lemma:prox}
For any $\lambda>0$ and $p\geq 1$, the proximity operator of $\lambda\JG$, where $\JG$ is defined in \eqref{penalty}, is given by
$$
 \textrm{prox}_{\lambda\JG} =\id-\pi_{\lambda\kk}.
 $$ 
where $\pi_{\lambda\kk}$ denotes the projection onto $\lambda\kk$, and $\kk$ is given by
\beeq{Kdef}{
\kk= \{x\in\rone^\dd, \nor{x}_{\gr_r,q}\leq 1,\  \textrm{for~} r=1,\dots,\ngr \}.
}
\end{lemma}

The proof exploits the particular definition of the penalty and relies on the Moreau decomposition
\beeq{eq:moreau}{\prox_{\lambda\Omega}(\w)=\w-\lambda\prox_{\frac{\Omega^*}{\lambda\,}}\left(\frac{\w}{\lambda}\right).}
Formula \eqref{Kdef} allows to compute the proximity operator of $\Omega$ starting from the proximity operator of the Fenchel conjugate. 
In our case, being $\JG$ one homogeneous, we obtain the identity minus the projection onto a closed and convex set.
The particular geometry  of $\kk$, which is the intersection of   $\ngr$ convex generalized cylinders ``centered'' on a coordinate subspace,  derives from definition of  $\JG$ and the explicit computation of its Fenchel conjugate. 
Observe that by definition $\JG$ is the infimal convolution of $\ngr$ functions, and precisely  the $\ngr$ norms on $\rone^{G_r}$ composed with the projections. By standard properties of the Fenchel conjugate, it follows that $(\JG)^*=\sum \iota_q$, where $\iota_q$ is the dual function of $\nor{\cdot}_p$, i.e.  the indicator function of the $\ell_q$ unitary ball in $\rone^{G_r}$. We give here a self-contained proof which does not use the notion of infimal convolution. A different proof for the case $p=2$ is given in \cite{obozinski:inria-00628498}. 
\begin{proof}
We start by computing explicitly the Fenchel conjugate of $\JG$.  By definition,
\begin{align*}
(\JG)^*(u) &=\sup_{\w\in\rone^\dd}\left[\,\,\langle\w,u\rangle- \min_{
 \substack{v\in\prod \rone^{G_r} \\P^*v= \w}}
 \sum_{r=1}^\ngr \nor{v_r}_p\right]
=\sup_{\w\in\rone^\dd}\left[\sup_{
 \substack{v\in\prod \rone^{G_r} \\P^*v= \w}}
\langle \w,u\rangle-  \sum_{r=1}^\ngr \nor{v_r}_p\right]\\
&=\sup_{v\in\prod \rone^{G_r}}
\left[\langle \sum_{r=1}^\ngr P^*_rv_r,u\rangle-  \sum_{r=1}^\ngr \nor{v_r}_p\right]=\sum_{r=1}^\ngr \sup_{v_r\in\rone^{G_r}}
 \left[\langle P_r^*v_r,u\rangle- \nor{v_r}_p\right]\\
 &\sum_{r=1}^\ngr \sup_{v_r\in\rone^{G_r}}
 \left[\langle v_r,P_r u\rangle- \nor{v_r}_p\right]=\sum_{r=1}^\ngr \iota_q(P_ru),
\end{align*}
where  $\iota_q$ is the Fenchel conjugate of $\nor{\cdot}_p$, i.e.  the indicator function of the $\ell_q$ unitary ball in $\rone^{G_r}$. We can rewrite the sum of indicator functions as $\sum_{r=1}^\ngr \iota_q(P_r u)=\iota_{\kk}(u).$
It is well-known that
\[
\prox_{\lambda \iota_{\kk}}(\w)=\pi_{\kk}(\w).
\]
Using the Moreau decomposition \eqref{eq:moreau} and basic properties of the projection we obtain
\beeq{eq:prdef}{\prox_{\lambda\Omega }(\w)=\w-\lambda\pi_{\kk}(\w/\lambda)=\w-\pi_{\lambda{\kk}}(\w).}
\end{proof}

The following lemma shows that, when evaluating the projection $\pi_{\kk}( x)$, we can restrict ourselves to a subset of {\itshape active} groups, denoted by $\hat{\mathcal{G}}= {\mathcal{G}(\hat x)}$ and defined in Lemma \ref{lemma:proj}. 
This equivalence is crucial to speed up  Algorithm \ref{algo:GSO}, in fact the number of active groups at iteration $\itext$ will converge to the number of selected groups, which is typically small if one is interested in sparse solutions.  
\begin{lemma}\label{lemma:proj}
Given $\w\in\rone^\dd$, it holds 
\beeq{primal_reduced}{\pi_{\lambda\kk} (\w)=\pi_{\lambda K^{\hat{\mathcal{G}}}_p}(\w)\,,}
where
$\hat{\mathcal{G}}:=\left\{\gr \in\mathcal{G},~\nor{\w}_{\gr,q}>\lambda \right\}.$
\end{lemma}
\begin{proof}
Given a group of indices $G$ and a number $p>1$, we denote by $C_{G,p}$  the convex set
\[
C_{G,p}=\{ \w\in\rone^\dd\,:\, \nor{\w}_{G,q}\leq 1\} .
\]
To prove the result we first show that for any subset $\mathcal{S}\subseteq\mathcal{G}$ the projection  onto the intersection $\lambda K_{p}^{\mathcal{S}}=\cap_{G\in \mathcal{S}}\lambda C_{G,p}$ is non-expansive coordinate-wise with respect to zero.
More precisely, for all $\w\in \mathbb{R}^d$, it holds that $|\pi_{\lambda K_{p}^{\mathcal{S}}} (\w)_i|\leq |\w_i|$ for all $i=1,\ldots,d$ and for all $\lambda>0$.
By contradiction, assume that there exists an index $\hat{\!j}$ such that $|\pi_{\lambda K_{p}^{\mathcal{S}}}( x)_{\hat{\!j}}|>| x_{\hat{\!j}}|$. Consider the vector $\tilde{\w}$ defined by setting
\[
\tilde{\w}_j=\begin{cases}  \pi_{\lambda K_{p}^{\mathcal{S}}}( x)_j & \text{if } j\neq \hat{\!j}\\ 
\w_{\hat{\!j}} & \text{otherwise.}   
\end{cases}
\]
First note that $\tilde{ x}\in  \lambda K^{\mathcal{S}}_p$, since $\nor{\tilde{ x}}_{G,q}\leq \nor{ \pi_{\lambda K_{p}^{\mathcal{S}}}( x)}_{G,q} \leq\lambda$ for all $G\in\mathcal{S}$. On the other hand
\[
\nor{ x-\tilde{ x}}^2=\sum_{\substack{j=1\\ j\neq \hat{\!j}}}^\dd ( x_j-\tilde{ x}_j)^2< \nor{ x -\pi_{\lambda K_{p}^{\mathcal{S}}}( x)}^2,
\]
which is a contradiction. 
To conclude, suppose that $\w\in \lambda K_{p}^{\mathcal{S}}$, with $\mathcal{S}\subseteq \mathcal{G}$. If we prove that 
\[
\proj{\lambda}(\w)=\pi_{\lambda K_p^{\mathcal{G}\setminus\mathcal{S}}}(\w),
\] 
we are done. For the sake of brevity denote $v=\pi_{\lambda K_p^{\mathcal{G}\setminus\mathcal{S}}}(\w)$.  Thanks to the non-expansive property it follows $|v_j|\leq |\w_j|$ for all $j=1,\ldots,\dd$ and therefore $v\in \lambda K_{p}^{\mathcal{S}}$. Since $v\in \lambda K_p^{\mathcal{G}\setminus\mathcal{S}}$ by hypotheses, we get that $v\in\lambda\kk$. Furthermore
by definition of projection
\[
\nor{v-\w}\leq \nor{w-\w}, \qquad \text{for every } w\in \lambda K_p^{\mathcal{G}\setminus\mathcal{S}}
\] 
and a fortiori  $\nor{v-\w}\leq \nor{w-\w}$ for every $w\in \lambda\kk$.
\end{proof}

\subsubsection{The projection on $\kk$ for general $p$}\label{sec:gen_proj}
The convex set $\kk$ is an intersection of convex sets, precisely
$$\kk = \bigcap_{\gr\in\mathcal{G}} C_{\gr,p}$$
where $ C_{\gr,p} = \{v\in\rone^\dd, \nor{v}_{\gr,q}\leq1\}$.

For general $p$ a possible minimization scheme for  computing the projection in \eqref{primal_reduced} can be obtained by applying the Cyclic Projections algorithm  \cite{boyle85} or one of its modified versions (see \cite{bauschke94ancora} and references therein).
In the particular case of  $p=2$, we describe the Lagrangian dual problem corresponding to the projection onto $K^{\mathcal{G}}_2$, and  we propose an alternative optimization scheme, the projected Newton method \cite{bertsekas82}, which better exploits the geometry of the set $K^{\mathcal{G}}_2$, and in practice proves to be  faster than the Cyclic Projections algorithm. 
Note that, in order to satisfy the hypothesis of Theorem \ref{conv_rate}, the tolerance for stopping the iteration must decrease with the outer iteration $\itext$. 

A simple way to compute the projection onto the intersection of convex sets is given by the 
{\itshape Cyclic Projections} algorithm \cite{boyle85}, which amounts to cyclically projecting onto each set.
Here  we recall in Algorithm \ref{algo:alt_proj} a modification of the Cyclic Projections algorithm proposed by \cite{bauschke94ancora}, for which strong convergence is guaranteed
(see Theorem 4.1 in \cite{bauschke94ancora}).
\begin{algorithm}[!ht]
\caption{Cyclic Projections}
\label{algo:alt_proj}
\begin{algorithmic}
\STATE \texttt{Given $ x\in\rone^\dd, \{C_{\gr_1,p},\dots,C_{\gr_\ngr,p}\}$}
\STATE \texttt{Let $\itint = 0$, $w^0 = \w$ and find $C_{\hat{G}_{1},p},\ldots,C_{\hat{G}_{\hat{B}},p}$}
\WHILE{\texttt{convergence not reached}}
\STATE $\itint = \itint+1$
\STATE \texttt{Let $\pi_\itint$ the projection onto $\tau C_{\hat{\gr}_{\itint{\,\rm mod}\hat{B}},p}$}
\STATE
$$
w^{\itint} = \frac{1}{\itint+1}\w + \frac{\itint}{\itint+1}\pi_\itint(w^{\itint-1}) 
$$
\ENDWHILE
\end{algorithmic}
\end{algorithm}

In the following we describe how to compute each projection $\pi_{C_{p,r}}$ for specific values of $p$.

\paragraph{$\mathbf{p=2}$.}
In this case $q=2$, and the projection is trivial
$$
[\pi_{\pone C_{\gr,2}}(w)]_j = 
\begin{cases}
\pone\frac{w_j}{\nor{w}_{\gr,2}}& \textrm{~if~} j\in\gr \textrm{~and~} \nor{w}_{\gr,2}>\pone\\
w_j &\textrm{~otherwise}
\end{cases}
$$

\paragraph{$\mathbf{p=\infty}$.}
In this case $q = 1$, and $C_{\gr,\infty}$ is an $\ell_1$ ball when restricting to the coordinates in $\gr$.
From Lemma 4.2 in \cite{fornasier2008}, we have that 
if $\nor{w}_{1}>\pone$, then the  projection of $w$ onto the $\ell_1$ ball of radius $\pone$, $\pone B_1$, is given by the {\itshape soft-thresholding operation}
$$
[\pi_{ \pone B_1}(w)]_j = (|w_j| - \mu)_+\mathrm{sign}(w_j)
$$
where $\mu$ (depending on $w$ and $\pone$) is chosen such that $\sum_j (|w_j| - \mu)_+=\pone$.

We recall a simple procedure provided  in  \cite{fornasier2008} for determining $\mu$.
In a first step, sort the absolute values of the components of $w$, resulting in the rearranged sequence, $w^*_j\geq w^*_{j+1}\geq 0$ for all $j$. 
Next, perform a search to find $k$ such that 
$$
\sum_{j=1}^{k-1} (w^*_j-w^*_k) \leq\pone\leq
\sum_{j=1}^k(w^*_j - w^*_{k+1}).
$$
Then set  $\mu = w^*_k + k^{-1}\left(\sum_{j=1}^{k-1}(w^*_j-w^*_k)-\pone\right)$

\paragraph{$\mathbf{p\neq 2,+\infty}$.} In these cases no known closed form for the projection  on the set $C_{G_r,p}$ exist, but it can be efficiently computed using Newton's method, as done in \cite{JacHamFad11}. 
\subsubsection{The projection on $\kk$ for $p=2$}
When $p=2$, the projection onto  $K_2^{\mathcal{G}}$ amounts to solving the constrained minimization problem \beeq{primal_2}{
\begin{array}{ll} 
\textrm{Minimize } &\nor{v-\w}^2\\
\textrm{subject to} &v\in \rone^d ,\   \nor{v}_{\gr,2}\leq \tau, \  ~\textrm{for} ~\gr\in\hat{\mathcal{G}},
\end{array}
}
which  Lagrangian dual problem can be written in a closed form.  Working on the dual is advantageous, since the number of groups  is typically much smaller than $\dd$, and furthermore Lemma \ref{lemma:proj} guarantees that one can restrict to the subset of groups 
\begin{equation}\label{eq:def_Gcap} \hat{\mathcal{G}}:=\{\gr \in\mathcal{G}:~\nor{\w}_{\gr,2}>\tau \}=:\{\hat{\gr}_1,\dots,\hat{\gr}_{\hat{\ngr}}\}\end{equation}
 which in general is a proper subset of $\mathcal{G}$.

In the following theorem we show how to compute the solution to problem \eqref{primal_2}, 
by solving the associated dual problem.

\begin{theorem}
\label{teo:K2}
Given $\w\in\rone^\dd$, $\mathcal{G}=\{\gr_r\}_{r=1}^\ngr$ with $\gr_r\subset  \{1,\dots,\dd\}$, $\hat{\mathcal{G}}$ as in \eqref{eq:def_Gcap} and $\tau>0$,
 the projection of $\w$ onto 
the convex set $\tau K_2^{\mathcal{G}}$ with $ K_2^{\mathcal{G}}= \{ v\in \rone^d :  \nor{v}_{\gr_r,2}\leq \tau ~\textrm{for}~r=1,\dots,\ngr\}$ is given by 
\beeq{eq:projection_2}{
\left[\pi_{\tau K_2^{\mathcal{G}}}(\w)\right]_j= \frac{\w_j}{1+\sum_{r=1}^{\hat{\ngr}} \lambda^*_r\mathbf{1}_{r,j}}\qquad \textrm{for } j= 1,\dots,\dd
}
where $\lambda^*$ is the solution of 
\beeq{dual_2}{
\argmax_{\lambda\in\rone^{\hat{\ngr}}_+}f(\lambda),\qquad \textrm{with }\quad
f(\lambda)= \sum_{j=1}^\dd\frac{-\w^2_j}{1+\sum_{r=1}^{\hat{\ngr}}\mathbf{1}_{r,j}\lambda_r}-\sum_{r=1}^{\hat{\ngr}}\lambda_r\tau^2,
}
and $\mathbf{1}_{r,j}$ equal to $1$ if $j$ belongs to group $\hat{\gr}_r$ and $0$ otherwise.
\end{theorem}
\begin{proof}
The Lagrangian function for the minimization problem \eqref{primal_2} is defined as
\begin{eqnarray}
L(v,\lambda)&=& \nor{v-\w}^2 +\sum_{r=1}^{\hat{\ngr}} \lambda_r(\nor{v}_{\gr_r}^2-\tau^2)
\nonumber\\
\!&\!=\!\!&\!\! \sum_{j=1}^\dd \left[ (v_j-\w_j)^2 + \sum_{r=1}^{\hat{\ngr} }\lambda_r \mathbf{1}_{r,j}v_j^2\right] - \sum_{r=1}^{\hat{\ngr}}\lambda_r\tau^2\nonumber\\
\!&\!=\!\!&\!\! \sum_{j=1}^\dd(1\!+\!\sum_{r=1}^{\hat{\ngr}}\!\mathbf{1}_{r,j}\lambda_r)\!\left(\!v_j\!-\!\frac{\w_j}{1\!+\!\sum_{r=1}^{\hat{\ngr}}\!\mathbf{1}_{r,j}\lambda_r}\right)^2\!\!\!-
 \sum_{j=1}^\dd\frac{\w^2_j}{1\!+\!\sum_{r=1}^{\hat{\ngr}}\!\mathbf{1}_{r,j}\lambda_r}\!-\sum_{r=1}^{\hat{\ngr}}\lambda_r\tau^2
 \! +\! \nor{ x}^2
 \nonumber\\
 \end{eqnarray}
 \noindent where $\lambda\in\rone^{\hat{\ngr}}$.
 The  dual function is then
 $$
f(\lambda) = \inf_{v\in\rone^\dd}L(v,\lambda) = 
 L\left(\frac{ x_j}{1+\sum_{r=1}^{\hat{\ngr}}\mathbf{1}_{r,j}\lambda_r},\lambda\right)=
 - \sum_{j=1}^\dd\frac{\w^2_j}{1+\sum_{r=1}^{\hat{\ngr}}\mathbf{1}_{r,j}\lambda_r}-\sum_{r=1}^{\hat{\ngr}}\lambda_r\tau^2 + \nor{ x}^2.
 $$
Since strong duality holds, the minimum  of (4) is equal to maximum of the dual problem
which is  therefore
 \beeq{dual_2_app}{
\begin{array}{ll}
\textrm{Maximize } &f(\lambda)\\
\textrm{subject to} &\lambda_r\geq0 ~\textrm{for} ~r = 1,\dots,{\hat{\ngr}}.
\end{array}
}
Once the solution $\lambda^*$ to the dual problem \eqref{dual_2_app} is obtained, the solution to the primal problem \eqref{primal_2}, $v^*$, is given by
$$
v^*_j = \frac{\w_j}{1+\sum_{r=1}^{\hat{\ngr}} \lambda^*_r\mathbf{1}_{r,j}}\qquad \textrm{for } j= 1,\dots,\dd.
$$
\end{proof}
The dual problem  can be efficiently solved, for instance, via Bertsekas' projected Newton method described in \cite{bertsekas82}, 
and here reported as Algorithm \ref{projection} in the Appendix, where the first and second partial derivatives  of $f(\lambda)$ are given by
$$
\partial_r f(\lambda)
 = \sum_{j=1}^\dd\frac{\w_j^2 \mathbf{1}_{r,j}}{(1+\sum_{s=1}^{\hat{\ngr}}\mathbf{1}_{s,j}\lambda_s)^2}-\tau^2 ,
$$
and
\begin{align*}
\partial_r\partial_{s}f(\lambda)
&=- \sum_{j=1}^\dd\frac{2\w_j^2\mathbf{1}_{r,j}\mathbf{1}_{s,j}}{(1+\sum_{s=1}^{\hat{\ngr}}\mathbf{1}_{s,j}\lambda_s)^3}\\& = 
\begin{cases}
0 & \textrm{if}~	\hat{\gr}_r\cap\hat{\gr}_{s} = \emptyset\\
-2 \sum_{j\in \hat{G}_{r}\cap \hat{G}_s}\w_j^2(1+\sum_{s=1}^{\hat{\ngr}}\mathbf{1}_{s,j}\lambda_s)^{-3}&\textrm{otherwise}.
\end{cases}
\end{align*}

Bertsekas' iterative scheme  combines the basic simplicity of  the steepest descent iteration \cite{rosen60}
with the quadratic convergence of the projected Newton's method \cite{brayton79}. It  does not involve the solution of a 
quadratic program thereby avoiding the associated computational overhead. Its convergence properties have been studied in \cite{bertsekas82} and are briefly mentioned in next section. 
\subsection{Convergence analysis of \ref{GSO} Algorithm }\label{sec:convergence}

In this subsection we clarify the accuracy in the computation of the projection which is required to prove convergence of the  Algorithm \ref{algo:GSO}. As mentioned above,  we rely on recent theorems providing a convergence rate for proximal gradient methods with approximations. 

\begin{definition}\label{def:appr_poj} We say that $w$ is an approximation of  $\proj{\pone/\sigma}(\w)$  with tolerance $\epsilon$ if 
$\|w-\proj{\pone/\sigma}(\w)\|\leq \epsilon$.\end{definition}

\begin{theorem}\label{teo:convergenceISTA}
Given $\w^0\in\rone^\dd$, and $ \sigma = ||\Psi^T\Psi||/n$. 
Assume that  $\proj{\pone/\sigma}(\w^\itext)$  in Algorithm \ref{algo:GSO} is approximately computed at  step $\itext$  with tolerance $\epsilon_\itext={\epsilon_0}/{\itext^{\alpha}}$.
\begin{itemize}
\item If $\alpha>2$, there exists a constant $C_I:=C_I(p,\mathcal{G},\w_0,\sigma,\pone,\alpha)$ such that the iterative update \eqref{ISTA} satisfies 
\beeq{conv_rate}{\phiv^p\left(\frac1 \itext\sum_{i=1}^\itext\w^i\right)-\phiv^p( x^*)\leq \frac {C_I}{\itext}. }
\item If $\alpha>4$, there exists a constant $C_F:=C_F(p,\mathcal{G},\w_0,\sigma,\pone,\alpha)$ such that the iterative update \eqref{FISTA} satisfies 
\beeq{conv_rate_f}{\phiv^p\left(\w^\itext\right)-\phiv^p( x^*)\leq \frac {C_F}{\itext^2}. }

\end{itemize}
\end{theorem}
\begin{proof}
It is enough to show that there exists a constant $C>0$ (independent of $w^\itint$ and $\w^\itext$) such that 
\begin{equation}\label{eq:err1}
\nor{w^\itint-\pi_{\tau K^{\mathcal{G}}_p}(\w^\itext)}\leq \frac{\epsilon_\itext}{C} \implies \Phi_{\frac\tau\sigma}(w^\itint) \leq \min\Phi_{\frac{\tau}{\sigma}}+ \epsilon_\itext
\end{equation}
where $\Phi_{\frac{\tau}{\sigma}}$ is defined as in \eqref{eq:def_prox}. Then the statement directly follows from Proposition 1 and Proposition 2 in \cite{Bach11}. 
In order to prove equation \eqref{eq:err1} first note that thanks to the assumption  $\cup _{r=1}^B G_r=\{1,\ldots,\dd\}$ made at the beginning, it easily follows from the definition that $\JG$ is a norm on $\mathbb{R}^\dd$, and therefore it is equivalent to the euclidean one. Thus,  there exists a constant $A$ (depending only on $p$ and $\mathcal{G}$) such that
\[
\JG( x)-\JG( x^\prime)\leq  A\nor{ x- x^\prime},\qquad \forall\,  x,  x^\prime \in \mathbb{R}^\dd.
\]   
Next, let $w$ and $\w$ be such that
\begin{equation}\label{eq:ass}
\nor{w-\proj{\pone/\sigma}(\w)}\leq \gamma,
\end{equation}
for some $\gamma>0$ (and suppose w.l.o.g. that $\gamma<1$). By Lemma \ref{lemma:prox} and by definition of $\prox_{\frac{\tau}{\sigma}\JG}(\w)$ and $\Phi_{\frac{\tau}{\sigma}}$ (see equation \eqref{eq:def_prox}) we have 
\[
\Phi_{\frac{\tau}{\sigma}}(\w-\proj{\pone/\sigma}(\w))=\min \Phi_{\frac\tau\sigma}.
\]
Thus, by equation \eqref{eq:ass}, and using the fact that $\JG$ is a norm
\begin{align*}
\Phi_{\frac{\tau}{\sigma}}( x-w)&=\frac{\sigma}{2\tau}\nor{w}^2+\JG(\w-w)\\
&\leq  \frac{\sigma}{2\tau}\nor{w-\proj{\pone/\sigma}(\w)}^2+ \frac{\sigma}{2\tau}\nor{\proj{\pone/\sigma}(\w)}^2+\frac{\sigma}{\tau}\langle w-\proj{\pone/\sigma}(\w), \proj{\pone/\sigma}(\w)\rangle \\ &\quad+\JG(\w-\proj{\pone/\sigma}(\w))+\JG(\proj{\pone/\sigma}(\w)-w)\\
&\leq \min \Phi_{\frac{\tau}{\sigma}}+\frac{\sigma}{2\pone}\gamma^2+\frac{\sigma}{\pone} \gamma \tilde{A}+A\gamma\\
&= \min \Phi_{\frac{\tau}{\sigma}}+ \left(\frac{\sigma}{2\tau}\gamma+\frac{\sigma}{\pone}\tilde{A}+A\right)\gamma\\
&\leq \min \Phi_{\frac{\tau}{\sigma}}+ C\gamma
\end{align*}
where $\tilde{A}$ is such that $\sup_{v\in K^{\mathcal{G}}_p} \nor{v}\leq \tilde{A}$ and $C=C(p,\mathcal{G},\sigma, \pone)$. Therefore, equation \eqref{eq:err1} holds with $C$ as defined above. 
\end{proof}
As it happens for the exact accelerations of the basic forward-backward splitting algorithm such as \cite{nesterov05,becker09, beck09}, convergence of the sequence $\w^\itext$ is no longer guaranteed unless strong convexity is assumed. \\

By Theorem 3.1 in \cite{bauschke94ancora}, Algorithm \ref{algo:alt_proj} is strongly convergent, and therefore, given arbitrary $\epsilon>0$ and $\w\in\mathbb{R}^\dd$, there exists an index $\itint_\itext:=\itint_\itext(\epsilon)$ such that  $w^{\itint_\itext}$ produced through Algorithm \ref{algo:alt_proj} enjoys the property
\[
\nor{w^{\itint}-\pi_{\tau K^{\mathcal{G}}_p}(\w^\itext)}\leq \epsilon,
\]
for every $\itint\geq \itint_\itext$.   

Algorithm \ref{algo:GSO} combined with Algorithm \ref{algo:alt_proj}  thus  converges to the minimum of \eqref{GSO} problem  with rate $1/\itext^2$, 
if the projection is approximately computed with tolerance $\epsilon_0/ \itext^\alpha$ with $\alpha>4$.
Similarly, one can use ISTA instead of FISTA as updating rule in Algorithm  \ref{algo:GSO}, obtaining 
the convergence rate  $1/m$, and setting $\alpha> 2$. It is clear that the choice of $\alpha$ defines the
stopping rule for the internal algorithm (see Subsection \ref{sec:cp_vs_dual}).\\

Every other algorithm producing admissible approximations can be used in place of Algorithm \ref{algo:alt_proj} in the computation of the 
projection. 
In the case $p=2$, we tested Bertsekas' projected  Newton method, reported in the Appendix as Algorithm \ref{projection}. Its convergence 
is not always guaranteed, since there are particular choices of $\w$ and $\mathcal{G}$ for which the partial Hessian of the dual function 
is not strictly positive defined, as would be required to ensure strong convergence (see Proposition 3 and Proposition 4 in \cite{bertsekas82}).  
However, ideas which are useful for circumventing the same problem for unconstrained Newton's method, such as preconditioning, could be 
easily adapted to this case, and convergence has  always been observed in our experiments (for more details see  the discussion in \cite{bertsekas82} and also the comments at the end of the next subsection).

\subsection{Computing the regularization path}

In Algorithm \ref{algo1} we report the complete scheme for computing the regularization path 
for the Group-wise Selection with Overlap  problem
 \eqref{GSO}, i.e. the set of
solutions corresponding to  different values of the  regularization parameter $\tau_1\!>\!\dots\!>\! \tau_{T}$. 
Note that we employ the {\em continuation} strategy proposed in \cite{hale08}.  
\begin{algorithm}[!ht]
\caption{Regularization path for GSO-$p$ }
\label{algo1}
\begin{algorithmic}
\STATE \textbf{Given:} $\tau_1>\tau_2>\dots>\tau_T, \mathcal{G},  \epsilon_0>0,\tolext>0$ \\
\STATE  \textbf{Let:}  $\sigma = ||\Psi^T\Psi||/n, ~\w^{(\tau_0)} = 0$
\FOR{$t = 1,\dots,T$}
\STATE \textbf{Initialize:} $ \w = \w^{(\tau_{t-1})}$
\WHILE{\texttt{convergence not reached}}
%
\STATE
$\bullet$ update $ x$ according to Algorithm \ref{algo:GSO}, with the projection computed 
via Cyclic Projections or by solving the dual problem\\
\ENDWHILE
\STATE $\w^{(\tau_t)} = \w$
\ENDFOR
\RETURN $\w^{(\tau_1)},\dots,\w^{(\tau_T)}$
\end{algorithmic}
\vspace{.3cm}
\end{algorithm}
%
When computing the proximity operator with  Bertsekas' projected Newton method,
a similar warm starting is applied to the inner iteration,  
since the $\itext$-th projection is initialized with the solution of the $(\itext\!-\!1)$-th projection.
Despite the local nature of Bertsekas' scheme, such an initialization empirically proved to guarantee convergence.   

\subsection{The replicates formulation}\label{sec:replication}

As discussed in Section \ref{sec:problem}, the most common method to solve \eqref{GSO} problem 
 is  to minimize the standard  group $\ell_1/\ell_p$ regularization (without overlap) 
in the expanded space of latent variables in \eqref{eq:repl} built by replicating variables belonging 
to more than one group, thus working in a $\tilde{\dd}$-dimensional space with $\tilde{\dd} = \sum_{r=1}^\ngr |\gr_r|$. 
Setting $\tilde{\Psi}=\Psi P^*$ and $\mathrm{R}^{\mathcal{G}}_p(v)=\sum_{r=1}^\ngr \nor{v_r}_p$, problem \eqref{eq:repl}
can be written as
\[
\min_{v\in \prod_{r=1}^\ngr \rone^{G_r}} \frac{1}{n} \nor{\tilde{\Psi} v-y}^2+2\pone  \mathrm{R}^{\mathcal{G}}_p(v).
\]

The main advantage of such a formulation relies on the possibility of using any state-of-the-art optimization procedure for $\ell_1$/$\ell_p$ regularization without overlap.  In terms of proximal methods, a possible solution is given by Algorithm \ref{algo1},
where the proximity operator can be now computed group-wise as
$$
\left((\textrm{prox}_{\lambda \mathrm{R}^{\mathcal{G}}_p}(v))_j\right)_{j\in{\gr}_r} = \left(I-\pi_{\lambda S_{{\gr}_r,p}}\right)\left((v_j)_{j\in{\gr}_r}\right)
$$
for all $r=1,\dots,\ngr$, where $S_{{\gr_r,p}}$ now denotes the $\ell_q$ unitary ball in $\mathbb{R}^{G_r}$.
Furthermore for $p=2$ and $p=+\infty$ each projection can be computed exactly as described in Subsection \ref{sec:gen_proj} , and the proximity operator of $\mathrm{R}^{\mathcal{G}}_p$ is thus exact.
The optimization algorithm for solving \eqref{GSO} via FISTA in the replicated space is reported in 
Algorithm \ref{algo:replication}.
\begin{algorithm}[!ht]
\caption{FISTA for Group-wise Selection without overlap}
\label{algo:replication}
\begin{algorithmic}
\STATE \textbf{Given:} $v^0\in\prod_{r=1}^B\rone^{G_r},  \tau>0,\sigma = ||\tilde{\Psi}^T\tilde{\Psi}||/\n$\\
\STATE \textbf{Initialize:} $\itext=0, w^1=v^0, t^1 = 1$\\
\WHILE{\texttt{convergence not reached}}
\FOR{$r=1,\dots,\ngr$}
\STATE
$$
v_r= \left(I-\pi_{\frac{\pone}{\sigma} S_{\gr_r,p}}\right)
\left(\left(w^{\itext}-\frac{1}{\n\sigma}
\tilde{\Psi}^T(\tilde{\Psi} w^\itext-y)\right)_{j\in {\gr}_r}
\right)
$$
\ENDFOR
$$
s_{\itext+1} = \frac{1}{2}\left(1+\sqrt{1+4s_\itext^2}\right)
$$
$$
w^{\itext+1} = \left(1+ \frac{s_{\itext}-1}{s_{\itext+1}}\right)v^\itext + \left(\frac{1-s_{\itext}}{s_{\itext+1}} \right)v_{\itext-1}
$$
\ENDWHILE
\RETURN $v^\itext$
\end{algorithmic}
\vspace{.3cm}
\end{algorithm}

The replicate formulation involves a much simpler proximity operator, but each iteration has higher computational cost, since now depends on $\tilde{\dd}$ rather than on $\dd$, 
and thus increases with the amount of overlap among variables subsets (see Section \ref{sec:expe} for numerical comparisons between the projection and replication approaches). 
\section{Numerical experiments}\label{sec:expe}


In this section we present  numerical experiments aimed at
studying the computational performance of the proposed family of optimization algorithms, 
and at comparing them with the state-of-the-art algorithms applied to the replicate formulation.

\subsection{Cyclic Projections vs dual formulation}\label{sec:cp_vs_dual}
We build $\ngr$ groups,$ \{\gr_r\}_{r=1}^B$, of size $\db$, with  $\gr_r\subseteq \{1,\dots,\dd\}$, 
by  randomly drawing sets of $\db$ indexes from $\{1,\dots,\dd\}$, and consider the cases $\db=10$, and $\db=100$.
We vary the number of groups $\ngr$, so that the dimension of the expanded space 
is $\ovlp$ times the input dimension, 
$\tilde{\dd} = \ovlp\dd$, with $\ovlp = 1.2, 2$ and $5$. 
Clearly this amounts to taking  $\ngr = \ovlp\cdot\dd/\db$. 
We then generate a vector $\w\in\rone^\dd$ by randomly drawing each of its entry 
from $\mathcal{N}(0,1)$.
We then  pick a value of $\pone$ such that, when computing $\textrm{prox}_{\pone\JG}(\w)$, 
all groups are active. 
Precisely 
 we take  $\pone = .8\cdot\min_{r=1,\dots,\ngr}\nor{\w}_{\gr_r,2}$. 
We first compute the {\itshape exact} solution 
$ x^\dagger = \textrm{prox}_{\Omega^{\mathcal{G}}_2}( x)$
\footnote{it is the solution computed via the projected Newton method for the dual problem with very tight tolerance}.
Then we compute the approximated solutions with the Cyclic Projections Algorithm \ref{algo:alt_proj} 
and by solving the dual via the projected Newton method. 
We will refer to the former as CP2 and to the latter as {\itshape dual}.
We stop the iteration when the distance from the exact solution is less than $\epsilon$ the norm of $ x^\dagger$. 
We consider different values for the tolerance $\epsilon$, precisely we take $\epsilon = 10^{-2},10^{-3},10^{-4}$.

Mean and standard deviation of the computing time over 20 repetitions are plotted 
in Figure \ref{fig:CP_vs_dual_p2_tempi_db10} and \ref{fig:CP_vs_dual_p2_tempi_db100}
for each value of $\alpha$ and $\epsilon$.
\begin{figure}
\includegraphics[width = \linewidth]{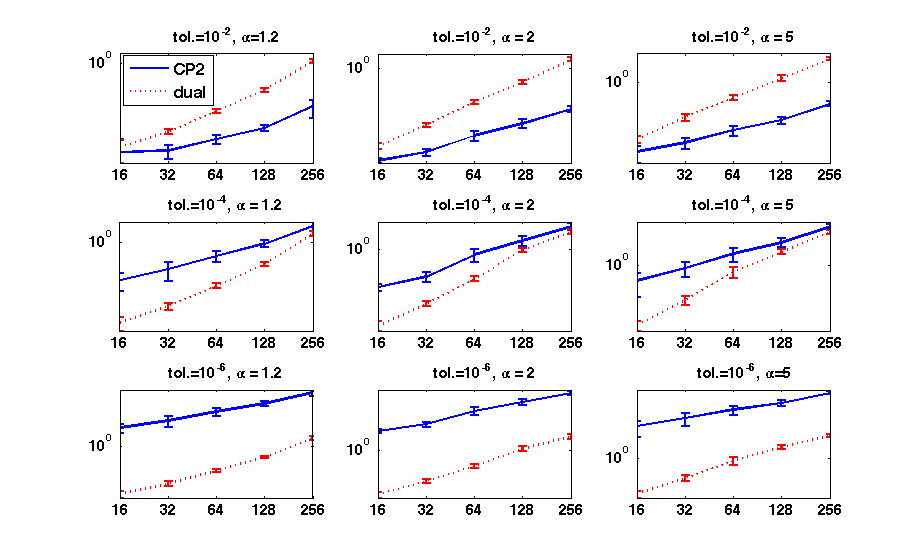}
\caption{Computing time (in seconds)  necessary for evaluating the prox  vs number of variables ($d$), 
for different values of the overlap degree $\alpha$ and the tolerance, for fixed group size $\db =10$.}
\label{fig:CP_vs_dual_p2_tempi_db10}
\end{figure}
\begin{figure}
\includegraphics[width = \linewidth]{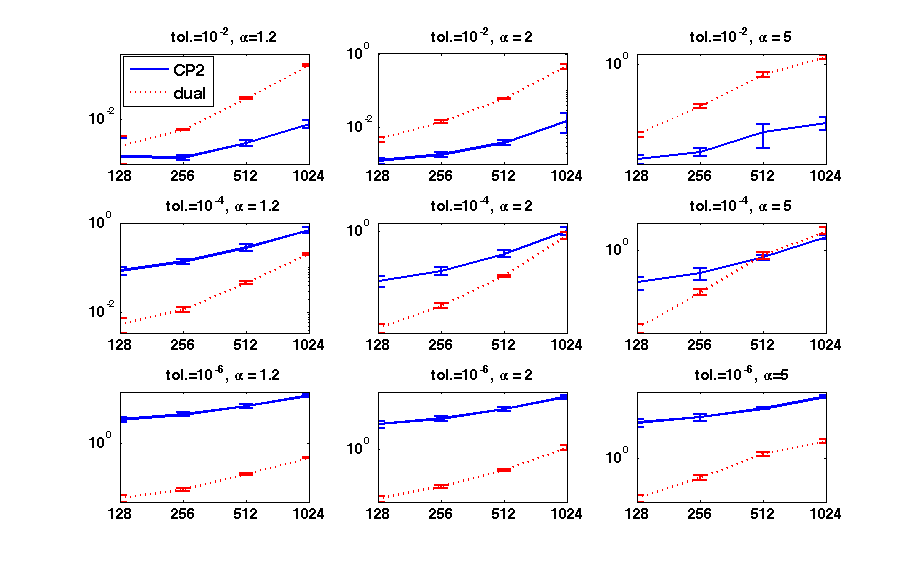}
\caption{Computing time (in seconds) necessary for evaluating the prox  vs number of variables ($d$), 
for different values of the overlap degree $\alpha$ and the tolerance, for fixed group size $\db =100$}
\label{fig:CP_vs_dual_p2_tempi_db100}
\end{figure}
The dual formulation is faster than the Cyclic Projections algorithm in most situations. 
It is convenient to use Cyclic Projections  when the number of active groups is high and the required tolerance very low.
When  computing the projection for Algorithm \ref{algo:GSO},
it is thus reasonable to use Cyclic Projections in the very first outer iterations, 
when the tolerance -- which depends on the outer iterations -- is low, and the solution could be not sparse, 
because still far from convergence. 
After few iterations, it is more convenient to resort to the dual formulation.
Even though, not optimal, in the following experiments, 
when denoting GSO-$2$ via projection we will consider 
always the projection computed with the dual formulation.

\subsection{Projection vs replication}
In this Subsection we 
compare the running time performance
of the proposed set of algorithms where the proximity operator is computed approximately, 
to state-of-the-art algorithms used to solve the equivalent formulation in the replicated space.
For such a comparison we restrict to $p=2$, 
since many benchmark algorithms are available 
in the case of groups that do not overlap.
In order to ensure a fair comparison, we first run some  preliminary experiments 
to identify the fastest codes for 
group $\ell_1$ regularization with no overlap.

\subsubsection{Comparison without overlap}

Recently there has been a very active research on this topic, see e.g. \cite{QinGol12,QinSchGol12,CheLinKim12}. For the comparison, we considered three algorithms which are representative  of the optimization techniques used to solve group lasso:
 interior-point methods, (group) coordinate descent and its variations, 
and proximal methods. 
As an instance of the first set of techniques we employed the publicly available Matlab code at \url{http://www.di.ens.fr/~fbach/grouplasso/index.htm}  described in \cite{bach08}. 
For coordinate descent methods, we employed the R-package \texttt{grlplasso}, 
which implements block coordinate gradient descent minimization for a set of possible loss functions. 
In the following we will refer to these two algorithms as ``IP'' and ``BCGD''.
Finally, as an instance of proximal methods, we use our Matlab implementation of FISTA for Group-wise Selection, 
namely Algorithm \ref{algo:replication} with FISTA instead of ISTA as updating rule. 
We will refer to it as ``PROX''.

We first observe that the solutions of the three algorithms coincide up to an error  which depends on each algorithm tolerance.
We thus need to tune the each tolerance  in order to guarantee that 
all iterative algorithms are stopped when the level of approximation to the 
true solution is the same.
Toward this end, we run Algorithm PROX with machine precision, $\tolext=10^{-16}$, 
in order to have a good approximation of the asymptotic solution. 
We observe that for many values of $\n$ and $\dd$, and over a large range of values of $\tau$, 
the approximation of  PROX when $\tolext\!=\!10^{-6}$
is of the same order of the approximation of  IP with $\texttt{optparam.tol}\! =\! 10^{-9}$, and of BCGD
with $\texttt{tol}\!=\!10^{-12}$. 
Note also that with these tolerances  the three solutions coincide also in terms of  selection, 
i.e. their supports  are identical for each value of $\tau$. 
Therefore the following results correspond to 
 $\texttt{optparam.tol} = 10^{-9}$ for IP,  $\texttt{tol}=10^{-12}$ for  BCGD,  and  $\tolext =10^{-6}$ for PROX.
For the other parameters of IP we  used the values used in the demos supplied with the code.\\
Concerning the data generation protocol, the input variables $x = (x_1,\dots,x_\dd)$ are uniformly drawn from $[-1,1]^{\dd}$.
The labels $y$ are computed using a noise-corrupted linear regression function, i.e.  $y = \w\cdot x + w$, where
$\w$ depends on the first $30$ variables, 
$\w _j = c$ if $j\!=\!1,\dots,30$, and $0$ otherwise,
 $w$ is an additive noise, $w\!\sim\! N(0,1)$, and $c$ is a rescaling factor that sets the signal to noise ratio to 5:1.
In this case the dictionary coincides with the variables, $\Psi_j(x) \!= \!x_j$ for $j\!=\!1,\dots,\dd$.
We then evaluate the entire regularization path for the three algorithms with $\ngr$ sequential groups of $10$ variables, 
($\gr_1\! \!=\!\! [1,\dots,10]$, $\gr_2 \!\!=\!\![11,\dots,20]$, and so on), for different values of $\n$ and  $\ngr$.
In order to make sure that  we are working on  the correct range of values for the  parameter
$\tau$, we first
evaluate the set of solutions of PROX corresponding to a large range of 500 values for $\tau$, with $\tolext \!=\! 10^{-4}$.
We then determine the smallest value of $\tau$ which corresponds to selecting less than $\n$ variables, $\tau_{min}$, 
and the smallest one returning the null solution, $\tau_{max}$. 
Finally we build the geometric series of $50$ values between $\tau_{min}$ and $\tau_{max}$, 
and use it to evaluate the regularization path on the three algorithms.
In order to obtain robust estimates of the running times, we repeat 20 times for each pair  $\n,\ngr$.\\
\begin{table}[t]
\caption{Running time (mean and standard deviation) in seconds for computing the entire regularization path of IP, BCGD, and PROX for different values of $\ngr$, and $n$.  }
\label{tab:GL-times}
\begin{center}
\begin{tabular}{l l}
$\n = 100$ & \begin{tabular}{c|p{2.3cm}|l}
 		& $\ngr = 10$ & $\ngr = 100$ \\
\hline
IP	&	$5.6\pm0.6$		&	$60\pm90$	\\
BCGD	&	$2.1\pm0.6$		&	$2.8\pm0.6$	\\
PROX	&	$0.21\pm0.04$	&	$2.9\pm0.4$	\\
\end{tabular} \\
&\\
 $\n = 500$&\begin{tabular}{c|p{2.3cm}|l}
 		& $\ngr = 10$ & $\ngr = 100$ \\
\hline
IP	&	$2.30\pm0.27$		&	$370\pm30$	\\
BCGD	&	$2.15\pm0.16$		&	$4.7\pm0.5$	\\
PROX	&	$0.1514\pm0.0025$	&	$2.54\pm0.16$	\\

\end{tabular}\\
&\\
$\n = 1000$&\begin{tabular}{c|p{2.3cm}|l}
 		& $\ngr = 10$ & $\ngr = 100$ \\
\hline
IP	&	$1.92\pm0.25$		&	$328\pm22$	\\
BCGD	&	$2.06\pm0.26$		&	$18\pm3$\\
PROX	&	$0.182\pm0.006$	&	$4.7\pm0.5$	\\
\end{tabular}
\end{tabular}
\end{center}
\end{table}
In Table \ref{tab:GL-times} we report the computational times required to evaluate the entire regularization 
path for the three algorithms.
Algorithms BCGD and PROX are always faster than IP
which, due to memory reasons, 
cannot be applied to problems where the number of variables are more than $5000$, since it requires 
to store the $\dd\times\dd$ matrix $\Psi\times\Psi$. It must be said that the code for GP-IL  was made available
 mainly in order to allow reproducibility of the results presented in \cite{bach08}, and is not optimized in terms of time and memory occupation.
However it is well known that standard second-order methods are typically  precluded on 
large data sets, since they need to solve large systems of linear equations to compute the 
Newton steps. PROX is the fastest for $\ngr = 10,100$ and has a similar behavior to BCGD.
The candidates as benchmark algorithms  for  comparison with
FISTA via projection are therefore BCGD and PROX. 
Since we are more familiar with the PROX algorithm, we therefore compare  FISTA via projection 
with the PROX algorithm, i.e.  FISTA via replication only.

\subsubsection{Comparison with overlap}
Here we compare two different implementations of the GSO-2 solution: 
FISTA via approximated projection computed by solving the dual problem with projected Newton method,
and FISTA via replication. We will refer to the former as FISTA-proj, and to the latter as FISTA-repl.\\

The data generation protocol is equal to the one described in the previous experiments, 
but $\w$ depends on the first $12/5\db$ variables (which correspond to the first three groups)
$$
\w = (\underbrace{c,~\dots,~c}_{\db\cdot12/5\textrm{ times}},\underbrace{0,~0,~\dots,~0}_{\dd-\db\cdot12/5\textrm{ times}}).
$$
We then define $\ngr$ groups of size $\db$, so that $\tilde{\dd} =\ngr\cdot\db>\dd$.
The first three groups correspond to the subset of relevant variables, and are defined as
 $\gr_1 = [1,\dots,\db]$, $\gr_2 = [4/5\db +1,\dots,9/5\db]$,
and $\gr_3 = [1,\dots, \db/5, 8/5\db +1,\dots,12/5\db]$, so that 
they have a  $20\%$ pair-wise overlap.
The remaining $\ngr-3$ groups are built by  randomly drawing sets of $\db$ indexes from $\{1,\dd\}$.
In the following we will let $\n = 10 |\gr_1\cup\gr_2\cup\gr_3|$, i.e. $\n$ is ten times the number of relevant variables, 
and vary $\dd, \db$.
We also vary the number of groups $\ngr$, 
so that the dimension of the space of latent variables is $\ovlp$ times the input dimension, 
$\tilde{\dd} = \ovlp\dd$, with $\ovlp = 1.2, 2,5$. Clearly this amounts to taking 
 $\ngr = \ovlp\cdot\dd/\db$. 
The parameter $\ovlp$ can be thought of as the average number of groups a single variable belongs to. 
We identify the correct range of values for $\tau$ as in the previous experiments, using FISTA-proj with loose tolerance,
and then evaluate the running time and the number of iterations necessary to compute the entire regularization path for
FISTA-repl on the expanded space and FISTA-proj, both with $\tolext = 10^{-6}$. 
Finally we repeat 20 times for each combination of the three parameters $\dd,\db$, and $\ovlp$.

\begin{table}[!h]
\caption{Running time (mean $\pm$ standard deviation) in seconds  for $\db\!=\!10$ (top), and $\db\!=\!100$ (below). 
For each $\dd$ and $\ovlp$, the left and right side correspond to  FISTA-proj, and FISTA-repl, respectively.}
\label{tab:GLO-times}
\begin{center}
\begin{tabular}{l|cc|cc|cc}
	& \multicolumn{2}{c|}{$\ovlp = 1.2$}  &\multicolumn{2}{c|}{$\ovlp=2$}  &\multicolumn{2}{c}{$\ovlp=5$}\\
\hline
$\dd\!=\!\!1000$	&	$0.15\pm0.04$	&	$0.20\pm0.09$	&	$1.6\pm0.9$	&	$5.1\pm2.0$	&	$12.4\pm1.3$	&	$68\pm8$\\
$\dd\!=\!\!5000$	&	$1.1\pm0.4$	&	$1.0\pm0.6$	&	$1.55\pm0.29$	&	$2.4\pm0.7$	&	$103\pm12$	&	$790\pm57$\\
$\dd\!=\!\!10000$	&	$2.1\pm0.7$	&	$2.1\pm1.4$	&	$3.0\pm0.6$	&	$4.5\pm1.4$	&	$460\pm110$	&	$2900\pm400$\\
\end{tabular}
\end{center}
\begin{center}
\begin{tabular}{l|cc|cc|cc}
	& \multicolumn{2}{c|}{$\ovlp = 1.2$}  &\multicolumn{2}{c|}{$\ovlp=2$}  &\multicolumn{2}{c}{$\ovlp=5$}\\
\hline
$\dd\!=\!\!1000$	&	$11.7\pm0.4$	&	$24.1\pm2.5$	&	$11.6\pm0.4$	&	$42\pm4$	&	$13.5\pm0.7$	&	$1467\pm13$\\
$\dd\!=\!\!5000$	&	$31\pm13$	&	$38\pm15$	&	$90\pm5$	&	$335\pm21$	&	$85\pm3$	&	$1110\pm80$\\
$\dd\!=\!\!10000$	&	$16.6\pm2.1$	&	$13\pm3$	&	$90\pm30$	&	$270\pm120$	&	$296\pm16$	&	--\\
\end{tabular}
\end{center}
\end{table}

\begin{table}[!h]
\caption{Number of iterations (mean $\pm$ standard deviation)  for $\db=10$ (top) and $\db=100$ (below). 
For each $\dd$ and $\ovlp$, the left and right side correspond to FISTA-proj, and FISTA-repl, respectively.}
\label{tab:GLO-iter}
\begin{center}
\begin{tabular}{l|cc|cc|cc}
	& \multicolumn{2}{c|}{$\ovlp = 1.2$}  &\multicolumn{2}{c|}{$\ovlp=2$}  &\multicolumn{2}{c}{$\ovlp=5$}\\
\hline
$\dd\!=\!\!1000$	&	$100\pm30$	&	$80\pm30$	&	$1200\pm500$	&	$1900\pm800$	&	$2150\pm160$	&	$11000\pm1300$\\
$\dd\!=\!\!5000$	&	$100\pm40$	&	$70\pm30$	&	$148\pm25$	&	$139\pm24$	&	$6600\pm500$	&	$27000\pm2000$\\
$\dd\!=\!\!10000$&	$100\pm30$	&	$70\pm40$	&	$160\pm30$	&	$137\pm26$	&	$13300\pm1900$&	$49000\pm6000$\\
\end{tabular}
\end{center}

\begin{center}
\begin{tabular}{l|cc|cc|cc}
	& \multicolumn{2}{c|}{$\ovlp = 1.2$}  &\multicolumn{2}{c|}{$\ovlp=2$}  &\multicolumn{2}{c}{$\ovlp=5$}\\
\hline
$\dd\!=\!\!1000$	&	$913\pm12$	&	$2160\pm210$	&	$894\pm11$	&	$2700\pm300$	&	$895\pm10$	&	$4200\pm400$\\
$\dd\!=\!\!5000$	&	$600\pm400$	&	$600\pm300$	&	$1860\pm110$	&	$4590\pm290$	&	$1320\pm30$	&	$6800\pm500$\\
$\dd\!=\!\!10000$	&	$81\pm11$	&	$63\pm11$	&	$1000\pm500$	&	$1800\pm900$	&	$2100\pm60$	&	--\\
\end{tabular}
\end{center}

\end{table}

Running times and number of iterations are reported in Table \ref{tab:GLO-times} and \ref{tab:GLO-iter}, respectively.
 When the overlap, that is $\ovlp$, is low the computational times of FISTA-repl and FISTA-proj are comparable.
As  $\ovlp$ increases, there is a clear advantage in using FISTA-proj instead of FISTA-repl. 
The same behavior occurs for the number of iterations.
\subsection{$p=2$ vs $p=\infty$}
We generate the groups and the coefficient vector as in Subsection  \ref{sec:cp_vs_dual}, with $\db = 10$.
Differently from the Subsection \ref{sec:cp_vs_dual}, here we compare the  computational performance 
of the same algorithm applied to two different problems: Cyclic Projections for $p=2$ and Cyclic Projections for $p=\infty$,
that yield different solutions, since 
 $\textrm{prox}_{\pone\Omega^{\mathcal{G}}_2}( x)\neq\textrm{prox}_{\pone\Omega^{\mathcal{G}}_\infty}( x)$. 
 In order to guarantee a fair comparison we consider two different values of $\pone$, 
 $\pone_2$ and $\pone_\infty$, such that, 
 when computing $\textrm{prox}_{\pone_2\Omega^{\mathcal{G}}_2}( x)$ and 
  $\textrm{prox}_{\pone_\infty\Omega^{\mathcal{G}}_\infty}( x)$, all groups are active. 
Precisely  we take  $\pone_2 = .8\cdot\min_{r=1,\dots,\ngr}\nor{\w}_{\gr_r,2}$. 
and $\pone_\infty = .8\cdot\min_{r=1,\dots,\ngr}\nor{\w}_{\gr_r,\infty}$. 
We compute the approximated solutions with the Cyclic Projections Algorithm \ref{algo:alt_proj}
for $p=2$ and $p=\infty$. We will refer to the former as CP2 and to the latter as CPinf.
We stop the iteration when the relative decrease of the approximated solution is below $\epsilon$.
We consider different values for the tolerance $\epsilon$, precisely we take $\epsilon = 10^{-2}, 10^{-3}, 10^{-4}$.

For each value of $\alpha$ and $\epsilon$ we estimate the number of iterations, and the computing time for the two algorithms, and average over 20 repetitions. Mean and standard deviation of number of iterations and the computing time are plotted in Figure \ref{fig:p2_vs_pinf_iter} and \ref{fig:p2_vs_pinf_tempi}.
In all conditions CP2 is much faster than CPinf.

\begin{figure}
\includegraphics[width = \linewidth]{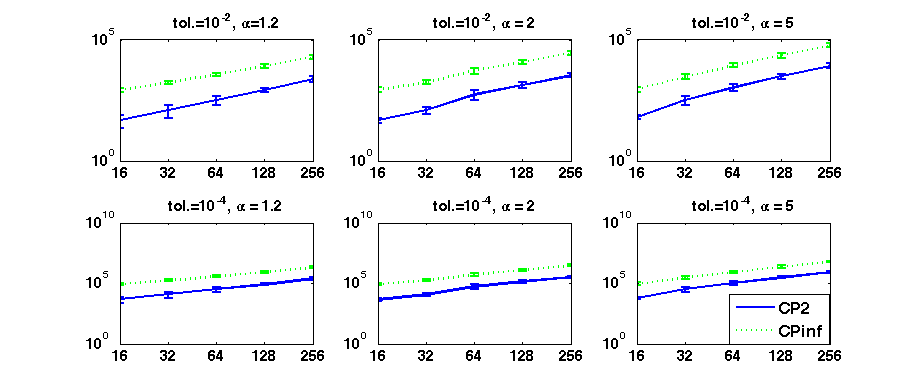}
\caption{Number of iteration necessary for evaluating the prox  vs number of variables ($d$), 
for different values of the overlap degree $\alpha$, and the tolerance.}
\label{fig:p2_vs_pinf_iter}
\end{figure}

\begin{figure}
\includegraphics[width = \linewidth]{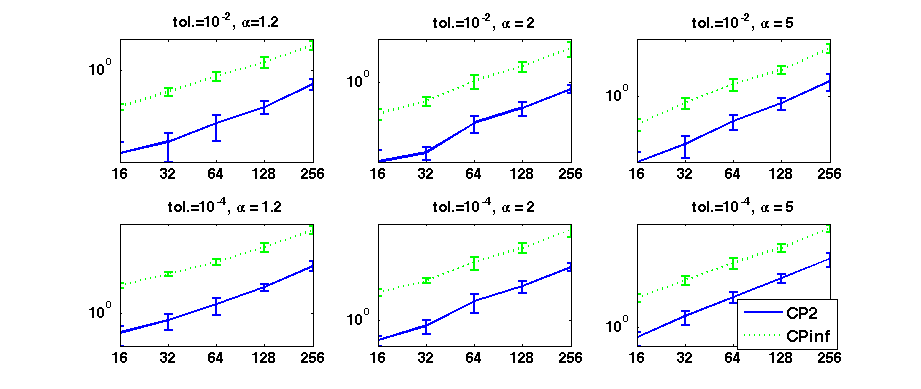}
\caption{Computing time (in seconds) necessary for evaluating the prox  vs number of variables ($d$), 
for different values of the overlap degree $\alpha$, and the tolerance}
\label{fig:p2_vs_pinf_tempi}
\end{figure}

\subsection{Real data Experiments: Microarray data}
In the previous subsection we have shown that, thanks to the computational efficiency of the
proposed projection algorithm, 
the GSO-$p$ regularization scheme  can be easily applied to large data sets with large group overlap.
Here we show that on real data, indeed, dealing with the entire data set without resorting to preprocessing
leads to improved prediction and selection performance.
We consider the  microarray experiment presented in \cite{jacob2009} 
where the breast cancer dataset compiled by \cite{vantveer2002} (8141 genes for 295 tumors) 
is analyzed with the group lasso with overlap penalty and the  637 gene groups corresponding to the MSigDB pathways  \cite{subramanian2005}.
In \cite{jacob2009} the accuracy of a logistic regression is estimated via 3-fold cross validation. On each split the 300 genes most correlated with the output are selected and the optimal $\tau$ is chosen via cross validation. 
$6,\,5$ and 78 pathways are selected with a $0.36\pm0.03$ cross validation error.
We applied FISTA-proj to the entire data set with two loops of k-fold cross validation (k$=\!3$ for testing). 
The obtained cross validation error is $0.33\pm0.05$ and $0.30\pm0.06$, with k$=\!3$ and k$=\!10$ for validation, respectively. 
In both cases  the number of selected groups is $2,3,$ and $4$, 
with 1 group in 3, and 3 pathways selected in 2 out of 3 splits.
The computing time for running the entire  framework for FISTA-proj 
(comprising data and pathways loading, recentering, selection via  FISTA-proj, regression via RLS on the selected genes,  and testing) 
is 850s (k$=\!\!3$) and 3387s (k$=\!\!10$).
Note that, while the improved cross validation error might be due to the second optimization step (RLS),
the improved stability is probably due to the absence of the preprocessing step, which can be highly unstable, 
thus compromising the overall stability of the solution.

\section{Discussion}
We have presented an efficient optimization procedure for computing the solution of a set of regularization
schemes that perform group-wise selection with overlapping groups, whose convergence is guaranteed. 
Our procedure allows dealing with high dimensional problems with large group overlap.
We have empirically shown that it has a significant computational advantage with respect to state-of-the-art
algorithms for group-wise selection applied on the expanded space built by replicating variables belonging to more than one group.
We also mention that computational performance may improve if
our scheme is used as core for the optimization step of active set methods, such as 
\cite{roth08}.
Finally, the improved computational performance enables to use
group-wise selection  with overlap  for pathway analysis of high-throughput biomedical data, 
since it can be applied to the entire data set and using all the information present in online databases,  
without pre-processing for dimensionality reduction.

\section*{Acknowledgments} Lorenzo Rosasco is assistant professor at DIBRIS, Universit\`a di Genova, Italy and currently on leave of absence.  The authors wish to thank Saverio Salzo for carefully reading the paper. 

\bibliographystyle{siam.bst}
\bibliography{tutta_la_bib}
\appendix
\section{Projected Newton Method}
In this appendix we report as Algorithm \ref{projection}  Bertsekas' projected Newton method described in \cite{bertsekas82}, with the modifcations needed to perform the maximization of a concave function instead of the minimization of a convex one. 

\begin{algorithm}[!ht]
\caption{Projection onto $K_2^{\mathcal{G}}$}
\label{projection}
\begin{algorithmic}
\STATE \textbf{Given:} $\w\in \rone^\dd,\lambda_{\textrm{init}}\in\rone^{\hat{\ngr}}, \eta\in(0,1),\delta\in(0,1/2),\epsilon>0$ \\
\STATE \textbf{Initialize:} $\itint=0,\lambda^0=\lambda^{\textrm{init}}$\\
\WHILE{($\partial_r f(\lambda^\itint)\neq 0$ \texttt{if} $\lambda_r >0$, \texttt{or} $\partial_r f(\lambda^\itint)>0$ \texttt{if} $\lambda_r =0$, \texttt{for some} $r=1,\dots,\hat{\ngr}$)}
\STATE $\itint := \itint+1$$$
\epsilon_\itint = \textrm{min}\{\epsilon,||\lambda^\itint-[\lambda^\itint+\nabla f(\lambda^\itint)]_+||\}
$$
$$
\mathcal{I}_+^\itint = \left\{r \,:\, 0\leq\lambda^\itint_r\leq\epsilon_\itint, ~\partial_r f(\lambda^\itint)<0\right\}
$$
\beeq{acca}{
H^\itint_{r,s} = \begin{cases}
0& \textrm{if}~r \neq s, \textrm{and}~r\in\mathcal{I}_+^\itint \textrm{or}~s\in\mathcal{I}_+^\itint\\
\partial_r\partial_{s} f(\lambda^\itint)& \textrm{otherwise}
\end{cases}
}

$$
\lambda(\alpha)= [\lambda^\itint -\alpha(H^\itint)^{-1}\nabla f(\lambda^\itint)]_+
$$
\STATE $m=0$
\WHILE{
$f(\lambda(\eta^m))-f(\lambda^\itint)<
\delta\left\{-\eta^m\sum_{r\notin\mathcal{I}_+^\itint}\sum_{s=1}^{\hat{\ngr}}\partial_r f(\lambda^\itint) [(H^\itint)^{-1}]_{r,s}\partial_s f(\lambda^\itint)+
\sum_{r\in\mathcal{I}_+^\itint}\partial_r f(\lambda^\itint)[\lambda_r(\eta^m)-\lambda^\itint_r]\right\}$}
\STATE $m:=m+1$
\ENDWHILE
$$
\lambda^{\itint+1} = \lambda(\eta^m)
$$
\ENDWHILE
\RETURN $\lambda^{\itint+1}$
\end{algorithmic}

\vspace{.3cm}
\end{algorithm}

The step size rule, i.e. the choice of $\alpha$, is a combination of 
the Armijo-like rule \cite{rosen60} and the Armijo rule usually employed in unconstrained 
minimization (see, e.g., \cite{poliak71}).

\end{document}